\documentclass[11pt]{amsart}
\usepackage{amsfonts,amssymb,amsmath,amsthm,latexsym,graphics,epsfig,amsfonts}
\usepackage{verbatim,enumerate,array,booktabs,color,bigstrut,prettyref,tikz-cd}
\usepackage{multirow}
\usepackage[all]{xy}
\usepackage{hyperref}
\usepackage[OT2,T1]{fontenc}
\usepackage{mathtools}
\usepackage{caption}
\usepackage{longtable}
\usepackage{mathtools}
\usepackage{tabularray}
\UseTblrLibrary{amsmath,varwidth}
\usepackage{tabularx}
\usepackage{longtable}
\usepackage{arydshln}
\makeatletter
\def\thickhline{%
	\noalign{\ifnum0=`}\fi\hrule \@height \thickarrayrulewidth \futurelet
	\reserved@a\@xthickhline}
\def\@xthickhline{\ifx\reserved@a\thickhline
	\vskip\doublerulesep
	\vskip-\thickarrayrulewidth
	\fi
	\ifnum0=`{\fi}}
\makeatother

\newlength{\thickarrayrulewidth}
\newtheorem{defn}{Definition}[section]

\newtheorem{corollary}[defn]{Corollary}
\newtheorem{lemma}[defn]{Lemma}

\newtheorem{cor}[defn]{Corollary}
\newtheorem{prop}[defn]{Proposition}

\theoremstyle{definition}

\newtheorem{remark}[defn]{Remark}

\newtheorem{example}[defn]{Example}

\newcommand{\C}{\mathbb C}

\newcommand{\Z}{\mathbb Z}
\newcommand{\F}{\mathbb F}

\begin{document}
	\author{Francesc Bars}
	\address{Departament de Matem\`atiques, Universitat Aut\`onoma de Barcelona, Catalunya\,
		\\
	}
	\email{brancesc.bars@uab.cat}
	
	\author{Joan-C. Lario}
	
	\address{Departament de Matem\'atiques, 
		Universitat Polit\'ecnica de Catalunya \\
		Barcelona, Catalunya }
	
	\email{joan.carles.lario@upc.edu}

    \author{Brikena Vrioni}
	
	\address{School of Arts and Sciences \\
		American International Univeristy, Kuwait}
	
	\email{brikena.vrioni@yahoo.com}

	\title{Diophantine stability for curves over finite fields}
	\date{\today}

 	\begin{abstract}
		We carry out a survey on curves defined over finite fields that are Diophantine stable; that is, with the property that the set of points of the curve is not altered under a proper field extension. First, we derive some general results of such curves and then we analyze several families of curves that happen to be Diophantine stable.
	\end{abstract}
 
	\maketitle
	\tableofcontents

	\section{Introduction}
	\label{intro}
	
	Let $V$ be an algebraic variety defined over a field $K$. Borrowing Mazur-Rubin, we say that $V$ has Diophantine stability for a proper field extension $L/K$ if $V(L)=V(K)$. In this paper, we restrict ourselves to the case of curves defined over finite fields. 
	Throughout, for curve  we mean a geometrically irreducible projective smooth algebraic variety of dimension 1.
	
	From now on, let $C$ be a curve defined over a finite field of size a prime power $q$ and genus $g\geq 1$. The curve $C/\F_q$ is a Diophantine stable curve (DS-curve for short) if there is $m>1$ such that $C(\F_{q^m})=C(\F_q)$. As usual, we shall denote its Hasse-Weil zeta function by:
	$$
	\zeta(C/\F_q,t) = 
	\operatorname{exp}
	\left( \sum_{m=1}^{\infty} 
	\frac{N_m}{m}  t^m  \right) = 
	\prod_{d=1}^\infty
	(1-t^d)^{-a_d} 
	$$
	where $N_m=\#C(\F_{q^m})$ and 
	$a_d$ denotes the number of closed points of $C$ of degree $d$.
	One has the relation 
	$$
	N_m = \sum_{d\mid m} d \, a_d \,,
	$$
	and using the Möebius inversion formula we also have
	$$
	a_d = \frac{1}{d} \sum_{d'\mid d} \mu(d') N_{d'}\,.
	$$
	Due to the Weil conjectures \cite{Kat76}, we can express
	$$
	\zeta(C/\F_q,t) = \frac{P(t)}{(1-t)(1-qt)}
	$$
	with 
	$P(t) = \prod_{i=1}^g 
	(1- \pi_i t)(1-\overline{\pi}_it )\in \Z[t]$ and the reciprocal of the complex roots satisfy  $|\pi_i|=q^{1/2}$. 
	We call $P(t)$ the Frobenius polynomial of $C$. It is also common to consider the corresponding monic reciprocal polynomial
	$$L(t)= \prod_{j=1}^{g} (t- \pi_i )(t-\overline{\pi}_i )\in \Z[t]\,.$$
	The polynomial $L(t)$ is called the Weil polynomial (or $L$-polynomial) of the curve, and their roots $\pi_{i}$ are
	called Weil $q$-numbers. The relation between them is
	$L(t)= t^{2g } P(1/t)$. 
	
	For future use, we also introduce the real Weil polynomial $h(x)$
	that has roots the real numbers $\mu_i=\pi_{i}+
	\overline{\pi}_{i}$ for $1\leq i\leq g$. The real Weil polynomial has degree $g$ and satisfies
	$$
	P(t) = t^{g} h\left( \frac{q t^2+1}{t}\right) \,.
	$$
	Given $P(t)$, we can find $h(x)$ by means of the $t$-resultant
	$
	\operatorname{Res}_t(q t^2+1-t x, P(t)- t^{g} h(x) ) \,,
	$
	and vice versa, given $h(x)$, we can find $P(t)$ by means of the $x$-resultant
	$
	\operatorname{Res}_x(q t^2+1-t x, P(t)- t^{g} h(x) ) \,.
	$
	
	In order for the curve $C$ to be a DS-curve we need analyze the numbers $N_m$ (or equivalently the numbers $a_d$). Since the number of points satisfy:
	$$
	N_m = \# C(\F_{q^m})= 1+q^m - \sum_{i=1}^{g} (\pi_i^m+\overline{\pi}_i^m)\,,
	$$
	it seems natural to explore Diophantine stability through the study of the Hasse-Weil zeta function.
	
	In the following section we obtain some basic results regarding Diophantine stability over finite finite fields, and the subsequent sections collect a number of families of DS-curves: low genus curves, Deligne-Lusztig curves, $M$-torsion Carlitz curves, and $M$-torsion Drinfeld curves.

 The authors want to thank E.\,Howe, R.\,Lercier, C.\,Ritzenthaler, and A.\,Sutherland for useful comments on a first draft of the manuscript. The second author also wants to thank B.\,Mazur, A.\,Quirós, J.-P.\,Serre, and R.\,Schoof, for inspiring conversations on the subject at the beginning of this work. Several of the results of this paper were obtained  in the context of supervising B.\,Vrioni's dissertation \cite{Vri}.
	
	\section{Basic properties}
	
	We begin with a result about the finiteness of (isomorphism classes) of DS-curves of a given genus~$g\geq 1$.
	
	\begin{prop}
		\label{interval}
		Let $C/\F_q$ be a curve of genus $g\geq 1$ such that $C(\F_q)=C(\F_{q^m})$ for some $m>1$. Then, the pair $(q,m)$ has to be chosen from a finite set depending only on $g$.
	\end{prop}

	\begin{proof} 
		Suppose that $C(\F_{q})=C(\F_{q^m})$ for some $m>1$. By the Hasse-Weil-Serre bound, one the one hand 
		we have
		that $N_1 =\# C(\F_{q})$ belongs to the interval centered in $1+q$ and radius 
		$g\,
		\lfloor 2\sqrt{q} \rfloor$:
		$$
		\vert N_1 - (1+q) \vert \leq 
		g\,
		\lfloor 2\sqrt{q} \rfloor \,.
		$$
		On the other hand, $N_m =\# C(\F_{q^m})$ belongs to the interval centered in $1+q^{m}$ and radius 
		$g\,
		\lfloor 2\sqrt{q^m} \rfloor$:
		$$
		\vert N_m - (1+q^{m}) \vert 
		\leq 
		g\,
		\lfloor 2\sqrt{q^m} \rfloor \,.
		$$
		Since we are assuming that $N_1=N_m$ with $m>1$, we need to show that the intersection of these two intervals is always empty except for a finite number of pairs $(q,m)$.
		Thus, let us show that if $q$ and $m$ are large enough, then one has
		$$
		(1+q) + g\,
		\lfloor 2\sqrt{q} \rfloor <
		(1+q^{m}) -  g\,
		\lfloor 2\sqrt{q^m} \rfloor\,.
		$$
		The above inequality is equivalent to 
		$$
		g( \lfloor 2\sqrt{q} \rfloor + 
		\lfloor 2\sqrt{q^m} \rfloor) < q^{m} - q
		$$
		which is clear to be true 
		when $q^m$ is large enough since $g$ is fixed. 
	\end{proof}
	
	Since the number of isomorphism classes of curves defined over a given number field is finite, we get the following consequence. 
	
	\begin{corollary}
		The number of isomorphism classes of DS-curves defined over finite fields of a given genus is finite.
	\end{corollary}

	A similar argument along the above lines allows us to obtain an immediate generalization for the case of algebraic varieties of higher dimension defined over finite fields. We left the proof to the reader.
	
	\begin{prop}
		Let $V$ be a non-singular projective variety of a given dimension~$d$  and given Betti numbers $\beta_i$ over a finite field. If $V$ is a DS-variety for $\F_{q^m}/\F_q$, then the pair $(q,m)$ is chosen
		from a finite set. 
		In particular, the number of isomorphism classes of $d$-dimensional non-singular projective varieties with prescribed Betti numbers that are DS-varieties is finite.
	\end{prop}
	
	Our second result concerns the determination of a list containing 
	all possible Frobenius polynomials~$P(t)$ for the numerators of the potential Hasse-Weil zeta functions attached to curves of fixed genus $g$ defined over a fixed finite field $\F_q$. Equivalently, we can ask for the list of possible candidate Weil polynomials $L(t)$ or the list of possible 
	candidate real Weil polynomials $h(x)$. In all three cases, the sequence of possible 
	number of places $[a_1,a_2,a_3,\dots,a_g]$ determines either $P(t)$, $L(t)$, and $h(x)$.
	For instance, writing 
	$$P(t)=\sum_{n=0}^{2g} A_n t^n$$ with
	$A_{k}= q^{k-g} A_{2g-k}$ for $k>g$, from the Taylor series in both sides of the formal identity
	$$
	\frac{(1-t)(1-qt)}
	{(1-t)^{a_1} (1-t^2)^{a_2} \dots 
		(1-t^g)^{a_g}}=
	P(t) \prod_{d\geq g}(1-t^d)^{a_d}\,,
	$$
	we can get polynomial expressions 
	$A_n=A_n(a_1,a_2,\dots,a_g)$ with rational coefficients. As we shall see, the coefficients $H_n(a_1,a_2,\dots,a_g)$
	of the potential real Weil polynomials
	$$
	h(x) = \sum_{n=0}^g H_{g-n}(a_1,a_2,\dots,a_g) x^n
	$$
	satisfy a remarkable property that helps us to determine the list of candidates in an efficient way. In any case, after getting the list of candidate polynomials, the hard task is to discard the cases that do not correspond to any curve. To this end, one can use the partial criteria of Serre and Howe \cite{Ser20}.
	
	There are at least two different methods to face the task to construct the list of 
	candidate sequences 
	$[a_1,a_2,\dots,a_g]$ for the number of places of curves of genus $g$ defined over $\F_q$. A first option is to make use of the Weil-Serre explicit formulas:
	one can choose a double-positive function $F(t)\gg 0$, 
	$F(t) = 1 + 2 \sum_{n\geq 1} c_n \, \cos(nt)$ such that $c_n=0$ for all $n>g$ and $c_1\neq 0$. Then one has
	$$
	\sum_{d\geq 2} d a_d \left( 
	\sum_{d\mid n} c_n q^{-n/2}
	\right) \leq 
	g + \sum_{n\geq 1} c_n q^{n/2} 
	+(1-a_1) 
	\sum_{n\geq 1} c_n q^{-n/2} \,,
	$$
	so that we get the list of all positive integers 
	$[a_1,a_2,a_3,\dots,a_g]$ satisfying the inequality above.
	Since all the coefficients of the unknown $a_d$ in the above inequality are non-zero for $d\leq g$, we can guarantee that the list of candidates $[a_1,a_2,a_3,\dots,a_g]$ is a finite list. The problem of this method is that one usually gets a huge list of candidates.
	
	An alternative route is as follows. We want to determine the list of candidate real Weil polynomials 
	$$
	h(x) = x^g + H_{1} x^{g-1} +\dots + H_{g-2} x^2 + H_{g-1} x+ H_g
	$$
	where each $H_n = H_n(a_1,\dots,a_g)$
	is a polynomial expression 
	in $\mathbb{Q}[a_1,\dots,a_g]$.
	
	\begin{prop}
		\label{recursionH}
		For every $1\leq i \leq g$, 
		the polynomial 
		$H_i(a_1,\dots,a_g)$ has multi-degree $(1,0,\dots,0)$ 
		in the variables 
		$(a_i,a_{i+1},\dots,a_g)$. In other words,
		we can write 
		$H_i=H_i(a_1,a_2,\dots,a_i)$ linearly in the variable $a_i$.
	\end{prop}
	
	\begin{proof} 
		First, we prove that the $2g$-degree polynomial 
		$P(t)=\sum_{n=0}^{2g} A_n t^n$ with indeterminate coefficients $A_n=A_n(a_1,a_2,\dots, a_g)$ satisfies the reverse statement. More precisely, we want to show that for every $k<g$ the polynomial $A_k=A_k(a_1,a_2,\dots, a_g)=
		A_k(a_1,a_2,\dots, a_k)$
		does not depend on $a_{k+1},\dots,a_g$ and it is of degree one in the variable $a_k$. We have $A_0(a_1,a_2,\dots, a_g)=1$ and $A_1(a_1,a_2,\dots, a_g)= a_1-(q+1)$.
		The polynomial 
		$A_k(a_1,a_2,\dots, a_g)$ is the $k$th coefficient 
		of the Taylor series of
		$$
		\begin{array}{l}
			\displaystyle{\frac{(1-t)(1-qt)}
				{(1-t)^{a_1} (1-t^2)^{a_2} \dots 
					(1-t^g)^{a_g}}}  =   \\[19pt]
			(1-t)(1-qt)(1-t)^{-a_1} (1-t^2)^{-a_2} \dots 
			(1-t^g)^{-a_g} = \\[13pt]
			\displaystyle{ (1-t)(1-qt)
				\left(
				\sum_{n=0}^\infty \binom{-a_1}{n}(-t)^{n}
				\right)
				\left(
				\sum_{n=0}^\infty \binom{-a_2}{n}(-t)^{2n}
				\right)
				\left(
				\sum_{n=0}^\infty \binom{-a_g}{n}(-t)^{gn}
				\right)
			}
		\end{array}
		$$
		where $\binom{\alpha}{n}$ denotes the generalized binomial number. To compute the coefficient of $t^k$ we only need to take care of the partial product
		$$
		(1-(q+1)t-q t^2) \prod_{i=1}^k 
		\left(
		\sum_{n=0}^\infty \binom{-a_i}{n}(-t)^{i n}
		\right).
		$$
		Therefore
		$A_k=A_k(a_1,a_2,\dots, a_g)=
		A_k(a_1,a_2,\dots, a_{k})$ does not depend on the variables $a_{k+1},\dots,a_g$. The 
		unique contribution of $a_k$ into the coefficient of $t^k$  occurs in
		$$
		(1-(q+1)t-q t^2)  
		\left(
		\sum_{n=0}^\infty \binom{-a_k}{n}(-t)^{k n}
		\right)\,
		$$
		and it is equal to $-\binom{-a_k}{1}= a_k$.
		
		The claim on the coefficients of the real Weil polynomial $h(x)$ follows from the equality 
		
		$$
		P(t) = t^g h\left( \frac{qt^2+1}{t}\right) \,.
		$$
		
		Indeed, letting $P(t)=\sum_{n=0}^{2g} A_n t^n$ with
		$A_{k}= q^{k-g} A_{2g-k}$ for $k>g$, 
		and $h(x)=\sum_{n=0}^g H_n x^{g-n}$ with $H_0=1$,
		one has
		$$
		\begin{array}{l}
			\displaystyle{
				\sum_{n=0}^{2g} A_n t^n} =
			\displaystyle{t^g 
				{\left(
					\sum_{n=0}^g H_n 
					{\left(
						\frac{qt^2+1}{t}
						\right)}^{g-n}
					\right)} = 
				\sum_{n=0}^g H_n (qt^2+1)^{g-n} t^n
				= }      \\[12pt]
			\displaystyle{\sum_{n=0}^g H_n 
				\sum_{k=0}^{g-n} \binom{g-n}{k} q^k t^{2k+n} = 
				\sum_{n=0}^{2g} 
				\left(
				\sum_{\substack{k=0\\ k\equiv n (2) }}^{n} H_{k}
				\binom{g-k}{(n-k)/2} q^{(n-k)/2} 
				\right)
				t^{n}
			}   \,. 
		\end{array}
		$$
		Hence, for $0\leq n\leq 2g$, it holds
		$$
		A_n =
		\sum_{\substack{k=0\\ k\equiv n (2) }}^{n} 
		H_{k} \binom{g-k}{(n-k)/2} q^{(n-k)/2}\,.
		$$
		The matrix of the corresponding linear system is lower triangular and 
  non-singular. Thus, the inverse matrix is upper triangular and it follows that $H_i=H_i(a_1,\dots,a_i)$ and of degree one in $a_i$ as desired, since we have proved the same property for the polynomials $A_i=A_i(a_1,\dots,a_i)$ for $i\leq g$ and 
		$A_{i}= q^{i-g} A_{2g-i}$ for $i>g$.
	\end{proof}

	The roots of $h(x)$ must be real and contained in the interval $[-2\sqrt{q},2\sqrt{q}]$. The same assertion holds also for the derivatives of $h(x)$ by Rolle's theorem. Then, to obtain the list of candidate sequences $[a_1,a_2,\dots,a_g]$, we can proceed by recursion as follows.
	
	
	Start with a candidate of length one $[a_1]$, with $a_1$ in the Hasse-Weil-Serre interval. By increasing $i$ from $2$ to $g$, suppose we have the list of partial candidate sequences of
	length $i-1$.
	For each one of the candidates $[a_1,a_2,\dots,a_{i-1}]$, 
	we substitute these values  into the $(g-i)$th derivative
	$$
	h^{(g-i)}(x) = T_i(x) + t(a_{i})
	$$
	where $T_i(x)$ is a $i$-degree polynomial in $\mathbb{Q}[x]$ with no constant term, and $t(a_i)$ is the constant term. That is,
	$t(a_{i}) = 
	(g-i)!\, H_i(a_1,\dots,a_{i-1},a_i)$ which is a linear 
	polynomial in $\mathbb{Q}[a_{i}]$ due to Proposition~\ref{recursionH}.
	
	Obviously, we do not know how to compute the roots of $h^{(g-i)}(x)$ since we do not know the value of $a_{i}$, but we can (and do) compute the roots of the derivative $T_i(x)'\in \mathbb{Q}[x]$. Let
	$\alpha_1,\dots , \alpha_{i-1}$ be the roots of $T_i(x)'$. By recursion, we know that all of them are real and belong to the interval $[-2\sqrt{q},2\sqrt{q}]$ since they are also the roots of $h^{(g-i+1)}(x)$. Let 
	$\alpha_0=-2\sqrt{q}$ and $\alpha_i=2\sqrt{q}$.
	For even~$i$, set  
	$$m=\displaystyle{
		\max_{j\textrm{ odd}}} \{ T_i(\alpha_j) \}\,,
	\qquad 
	M=\displaystyle{
		\min_{j\textrm{ even}}} \{ T_i(\alpha_j) \}\,.$$
	For odd $i$, set
	$$m=\displaystyle{
		\max_{j\textrm{ even}}} \{ T_i(\alpha_j) \}\,,
	\qquad 
	M=\displaystyle{
		\min_{j\textrm{ odd}}} \{ T_i(\alpha_j) \}\,.$$
	Finally, it remains to solve the linear inequalities in integers
	$$
	0\leq a_i \,,\quad 
	t(a_{i}) \leq \min\{M,|m|\}
	$$
	when $i$ is even, or 
	$$
	0\leq a_i \,,\quad 
	t(a_{i}) \leq \min\{|M|,m\}
	$$
	when $i$ is odd, 
	since we want the translates 
	$h^{(g-i)}(x) = T_i(x) + t(a_{i})$ to have all the roots in $[-2\sqrt{q},2\sqrt{2}]$.
	Every solution $a_i$ contributes to enlarge the list of partial sequence of candidates with $[a_1,a_2,\dots,a_{i-1},a_i]$.

	\vskip 0.25truecm
	
	\noindent Let us illustrate the above procedure with an example. 
	
	\vskip 0.25truecm
	
	\noindent{\bf The elephant silhouette}. Take genus $g=5$ and finite field of size $q=2$. Assume we start with $[a_1]=[9]$. Formally, the real Weil polynomial is given by
	$$
	\begin{array}{cl}
		h(x) = & x^5 + 6 x^4 + 
		(10+ a_2)\, x^3 + 
		(6 a_2+ a_3)\, x^2 + \\[8pt]
		& \displaystyle{\frac{1}{2} \left(a_2^2+29 a_2+12 a_3-20 +2 a_4\right) x + }\\[14pt]
		& \left(3 \,a_2^2+ a_2 a_3+27 a_2+16 a_3+6 a_4-12 +a_5 \right).
	\end{array}
	$$
	
	\noindent The fourth derivative
	$h^{(4)}(x)=  24 (6 + 5 x)$ has
	root $\alpha=-6/5$, and the third derivative is
	$$h^{(3)}(x)= 
	6 ( 10 x^2 + 24 x) +
	6 (10 + a_2)\,.$$
	Thus $T_2(x)=6 ( 10 x^2 + 24 x)$
	and $T_2(-6/5)=-432/5$. Hence, we want 
	$$6 (10 + a_2) \leq 432/5 = 86.4$$
	which amounts to $a_2\leq 4$.
	At this point, our list of partial 2-length candidates $[a_1,a_2]$ is given by  $[9,0]$, $[9,1]$, $[9,2]$, $[9,3]$, and $[9,4]$.
	Keep going on the procedure, at some point we get the partial candidate $[9,0,0,2]$. 
	A priori, since we know that
	$N_5 = a_1 + 5 a_5 \leq 2^5+1 + 2 g \sqrt{2^5}
	= 89.5685$, we must have $a_5\leq 16$.
	But our strategy performs better. The real Weil polynomial attached to the sequence $[9,0,0,2,a_5]$ is
	$$
	h(x) =x^5+6 x^4+10 x^3-8 x+a_5 \,.
	$$
	In the figure below, the elephant silhouette corresponds to the plot of the polynomial $T_5(x)=x^5+6 x^4+10 x^3-8 x $ and it suggests that the possible values of $a_5$ (if any) are very limited once we have obtained the previous values $[a_1,a_2,a_3,a_4]$.
	Indeed, one has that the unique polynomial $h(x)$
	obtained as a translation by positive integers from $T_5(x)$ that has the five reals roots in $[-2\sqrt{2},2\sqrt{2}]$ is achieved by $a_5=0$. 
	
	\begin{figure}[h]
		\centering
		\includegraphics[width=0.65\textwidth]{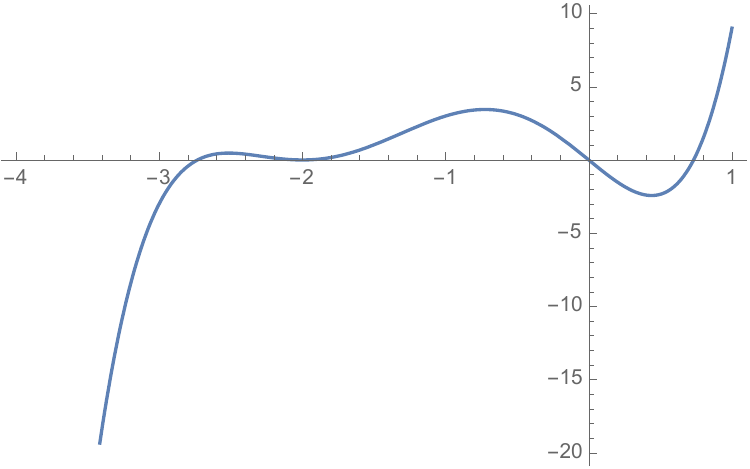}
		\label{fig:my_label}
		\caption{The real Weil polynomial
			$
			h(x) =x^5+6 x^4+10 x^3-8 x \,.
			$}
	\end{figure}

	This example corresponds 
	to the DS-curve of genus $5$ over $\mathbb{F}_2$
	with $[a_1,a_2,a_3,a_4,a_5]=[9,0,0,2,0]$
	given by the affine equation
	$$
	y^4+(x^2+x+1)y^2+(x^2+x)y+x^7+x^3 = 0 \,.
	$$
	The corresponding Frobenius polynomial is: 
	$$P(t)=32 t^{10}+96 t^9+160 t^8+192 t^7+184 t^6+144 t^5+92 t^4+48 t^3+20 t^2+6 t+1 \,.$$
	
	The above tree-type algorithm can be useful to speed the searching for DS-curves since one is looking for curves with certain number of places $a_d$ equals zero.  Regarding the procedure to compute all of the Frobenius polynomials with a given value of $a_1$, we must mention the earlier
 works of Smith \cite{Smi84} that uses this basic technique to enumerate all monic polynomials in $\Z[x]$ with a given trace and with all roots positive, and the idea goes back at least to Robinson \cite{Rob64}. Also, Lauter \cite{Lau00} spells out the procedure in the context of enumerating possible Frobenius polynomials for a curve over a given finite field with a given number of points.
	
	\begin{prop}\label{prop2.5}
		Let $C/\F_q$ be a curve of genus $g\geq 1$. Then one has $C(\F_q)=C(\F_{q^m})$ for some $m>1$ if and only if for every integer $d$ with $1<d\mid m$ it holds $a_d=0$.
	\end{prop}
	
	\begin{proof}
		It follows directly from $a_d\geq 0$ and the fact
		$
		N_m = \sum_{d\mid m} d \, a_d =a_1 = N_1 $.
	\end{proof}
	
	As a final remark in this section, we show the existence of curves with some $a_d=0$ that are not DS-curves. For example, consider the
	hyperelliptic curve of genus $g=6$ over the finite field of size $q=2$ defined by 
	$$y^2 + (x^6 + x^5 + x^4 + x^3 + x^2 + x + 1)\,y = x^{13} + x^5 + x + 1\,.$$
	Its 
	Frobenius polynomial is
	$$
	P(t)=
	(t^2 - t + 2)  (t^2 + t + 2)  (t^4 - t^3 - t^2 - 2t + 4)  (t^4 + t^3 - t^2 + 2t + 4)\,.
	$$
	The sequence of number of points over $\F_{2^m}$, for $m=1,..,16$ is
	$$
	[N_1,N_2,\dots,N_{16}]=[ 3, 5, 9, 17, 33, 11, 129, 257, 513, 1025, 2049, 4379, 8193, 16385, 32769,65537  ]\,,
	$$
	and the sequence of number of places for $m=1,..,16$ is
	$$
	[a_1,a_2,\dots,a_{16}]=[ 3, 1, 2, 3, 6, 0, 18, 30, 56, 99, 186, 363, 630, 1161, 2182 ,4080]\,.
	$$
	By using Proposition \ref{interval}, one can show that this curve is not a DS-curve.
	
	\section{Low genus curves}
	
	The following tables include the set of admissible pairs $(q,m)$ for which a 
	DS-curve for $\F_{q^m}/\F_{q}$ 
	of genus $g\leq 5$ can exist.
	These values are obtained via Proposition \ref{interval}.
	
	\label{admi}
	$$
	\begin{array}
		{|c|l|}
		\multicolumn{2}{c}{g=1}\\[3pt]
		\hline
		q & m \\
		\hline
		2 & 2, 3\\
		3 & 2 \\
		4 & 2 \\
		\hline
	\end{array}
	\qquad \qquad
	\begin{array}
		{|c|l|}
		\multicolumn{2}{c}{g=2}\\[3pt]
		\hline
		q & m \\
		\hline
		2 & 2, 3, 4\\
		3 & 2,3 \\
		4 & 2 \\
		5 & 2 \\
		\hline
	\end{array}
	\qquad \qquad
	\begin{array}
		{|c|l|}
		\multicolumn{2}{c}{g=3}\\[3pt]
		\hline
		q & m \\
		\hline
		2 & 2, 3, 4,5\\
		3 & 2,3 \\
		4 & 2,3 \\
		5 & 2 \\
		7 & 2 \\
		8 & 2 \\
		9 & 2 \\
		\hline
	\end{array}
	$$
	
	\vskip 0.35truecm
	
	$$
	\begin{array}
		{|c|l|}
		\multicolumn{2}{c}{g=4}\\[3pt]
		\hline
		q & m \\
		\hline
		2 & 2, 3, 4, 5, 6\\
		3 & 2, 3, 4 \\
		4 & 2, 3 \\
		5 & 2 \\
		7 & 2 \\
		8 & 2 \\
		9 & 2 \\
		11 & 2 \\
		\hline
	\end{array}
	\qquad \qquad
	\begin{array}
		{|c|l|}
		\multicolumn{2}{c}{g=5}\\[3pt]
		\hline
		q & m \\
		\hline
		2 & 2, 3, 4, 5, 6\\
		3 & 2,3, 4 \\
		4 & 2,3 \\
		5 & 2, 3 \\
		7 & 2 \\
		8 & 2 \\
		9 & 2 \\
		11 & 2 \\
		13 & 2 \\
		\hline
	\end{array}
	$$
	We do not claim that 
	all values in the above tables are necessarily attained by some DS-curve. 
	For instance, this is the case for
	$g=2$ and the admissible pair $(4,2)$ as we shall see.
	
	\subsection{Genus 1}
	\label{genus1} 
	
	The case of elliptic curves is by far the easiest. In the following table (and successive), 
	the first column displays the admissible pairs $(q,m)$ relative to~$g$ for which there exist DS-curves; the second column shows 
	defining equations of (representatives of the isomorphism classes of) DS-curves for $\F_{q^m}/\F_q$.
	The third column indicates the  number of  points
	$N=\# C(\F_q)=\# C(\F_{q^m})$.
	
	\begin{prop}
		The following table displays the set of all (isomorphism classes of) genus one DS-curves.
		$$
		\begin{array}{|l|l|c|}
			\hline
			(q,m)   &  \multicolumn{1}{|c|}{C}  & N \\
			\hline
			(2,2) & y^2+y = x^3+x  & 5 \\
			\hline
			(2,3) & y^2+y=x^3+1  & 4 \\
			& y^2+y = x^3+x  & 5 \\
			\thickhline
			(3,2) & y^2 = x^3+2\,x+1  & 7 \\
			\thickhline
			(4,2) & y^2+y = x^3 &  9 \\
			\hline
		\end{array}
		$$
	\end{prop}
	
	\begin{proof} 
		For every admissible pair $(q,m)$ one proceeds by inspection of the isomorphism classes of
		elliptic curves over $\F_q$ which can be listed easily.
	\end{proof}
	
	\begin{remark}
		Notice that the elliptic curve 
		$C: y^2+y = x^3$ in the last row
		is in fact defined over $\F_2$ (hence also over $\F_4$) and
		satisfies $\#C(\F_4)= \#C(\F_{16}) = 9$
		but $\#C(\F_2)=3$.
	\end{remark}
	
	\subsection{Genus 2}
	\label{genus2} 
	
	As for genus-2 curves over finite fields $\F_q$, the list of isomorphism classes 
	is practicable for small values of $q$ so that we can get easily the sublist of DS-curves for that genus.

	\begin{prop}
		The following table displays all the isomorphism classes of DS-curves of genus $g=2$ along with the admissible pairs $(q,m)$.
		$$
		\begin{array}{|l|l|c|}
			\hline
			(q,m)   &  \multicolumn{1}{|c|}{C} & N \\
			\hline
			(2,2) & y^2 +(x^2+x)\,  y= x^5+x^3+x^2+x  & 3 \\
			& y^2 +x\,  y= x^5+x  & 4 \\
			& y^2 + y= x^5+x^3  & 5 \\
			& y^2 +(x^3+x+1)\,  y= x^5+x^4+x^3+x  & 6 \\
			\hline
			(2,3) & y^2 +y= x^5+x^3+1  & 1  \\
			& y^2 +x\, y= x^5+x^2+x  & 2  \\
			& y^2 +y= x^5+x^4  & 5  \\
			\thickhline
			(3,2) & y^2= x^5 + 2\, x^4 + 2\,x^3 + 2\,x  &  5 \\
			\hline
			(3,3) & y^2= x^6 + x^4 + x^2 + 1  &  8 \\
			\thickhline
			(5,2) & y^2= x^5+4\,x   &  6 \\
			\hline
		\end{array}
		$$
	\end{prop}
	
	\begin{proof}
		To get the table, 
		one can build first the list of isomorphism classes of hyperelliptic curves of genus $g=2$ defined over $\F_q$ for the needed values of $q$ fairly easy and then select the DS-curves in the list. Alternatively, one can use directly the database on isomorphism classes of curves of small genus over finite fields elaborated by Sutherland \cite{sutherlandwebpage}.  
	\end{proof}
	
	\begin{remark}
		According to the tables 
		at the begging of the present section, a priori the cases 
		$(q,m)=(2,4)$ and $(4,2)$ have a chance to appear for genus-$2$ DS-curves. Both cases are excluded by inspection of the 
		representatives of isomorphism
		classes of curves.
		For instance, in the case $(q,m)=(4,2)$,
		it turns out that
		the intersection of the Hasse-Weil intervals for $N_1$ and $N_2$ is $[7,10]$, but 
		there are not genus-$2$ curves $C$ over $\F_4$ with
		$N=\# C(\F_4)=\# C(\F_{16})$ for $N=7$, $8$, $9$ or $10$.
		The minimal 
		difference
		$\# C(\F_{16})-\# C(\F_{4})$ among the genus-$2$ curves defined over $\F_4$
		turns out to be $2$ and it is attained by the curve
		$$
		y^2 +(x^2+x) \, y= 
		\alpha \, (x^5 + x^3 + x^2 +x) \,.
		$$
		In the equation above and hereafter, 
		for non-prime fields we let $\alpha$ denote a
		Conway generator of the finite field $\F_q=\F_p(\alpha)$; that is, $\alpha$ is a root of the Conway polynomial defining the extension $\F_q/\F_p$ where $p$ is the prime characteristic.
	\end{remark}

	\subsection{Genus 3}
	\label{genus3} 
	
	For genus-3 curves, things begin to get more intricate. Still we can make use of Sutherland's database. However, the database does not cover yet all the isomorphism classes of genus-$3$ curves defined over the finite fields for all the cases with potential presence of Diophantine stability.
	To be more precise, 
	from Sutherland's database,
	we lack the following isomorphism classes of genus-$3$ curves:
	
	\begin{itemize}
		\item Hyperelliptic curves over $\mathbb{F}_4$;
		\item Hyperelliptic curves over $\mathbb{F}_8$;
		\item Non-hyperelliptic curves over $\mathbb{F}_7$;
		
		
		\item Non-hyperelliptic curves over $\mathbb{F}_8$.
	\end{itemize}
	
	Luckily, we shall be able 
	to either justify the absence of DS-curves 
	or to find the ones with Diophantine stability in the isomorphism classes  under-construction in Sutherland's database. Hence, we can (and do) provide the complete list of DS-curves of genus $3$.

	\begin{prop}
		\label{tablesgenus3}
		The following tables display all the genus-3 DS-curves over finite fields.
		We first display the hyperelliptic DS-curves followed by the plane quartics defining equations for the non-hyperelliptic DS-curves.
	\end{prop}
	
	$$
	\begin{array}{|l|l|c|}
		\hline
		(q,m)   &  \multicolumn{1}{|c|}{C}  & N \\
		\hline
		(2,2)  & y^2 +(x+x^2)\, y  = 
		x^7 + x^6 + x^5  + x  & 3 \\
		& y^2 +(x+x^2)\, y = 
		x^7 + x^6 + x^5 + x^4  + x^2  + x   & 3 \\
		& y^2+x\, y =  x^7 + x^6 + x^2  + x   & 4 \\ 
		& y^2+x\,y = x^7 + x^6 + x^5  + x   & 4 \\ 
		& y^2+(x^2+x^4)\,y = x^5 + x^4 + x^3  + x   & 4 \\ 
		& y^2+y = x^7 + x^6   & 5 \\ 
		& y^2 + (x^4+x^2+x+1)\,y = x^7 + x^5  + x^4 + x^3   & 5 \\ 
		& y^2+(x^4+x^2+x)\,y= x^6  + x^3  + x^2   + x    & 5 \\ 
		& y^2+(x^4+x+1)\,y= 
		x^7 + x^5  + x^4 + x^3 + x^2  + x     & 6 \\ 
		& &   \\[-8pt]
		& x^{4} + x^{3} y + y^{4} + x^{2} y  + y^{3}  + x  + 1 &  0 \\         
		& x^{4} + x^{3} y + y^{4} + x^{3}  + x    &  1 \\
		& x^{4} + x y^{3} + y^{4} + x^{3}  + x^{2} y  + x y^{2}  + x   &  1 \\
		& x^{3} y + x y^{3} + y^{4} + x^{2} y  + x    &  2
		\\
		& x^{3} y + x y^{3} + y^{4} + x y^{2}  + x    &  2
		\\
		& x^{4} + x^{3} y + x y^{3} + x^{3}  + x y^{2}  + y^{2}  + x   &  2 \\
		& x^{4} + x^{2} y^{2} + x y^{3} + x^{3} + x^{2} y  + y^{2} 
		+ x    &  2 \\
		& x^{3} y + x y^{3} + y^{4} + x^{3}  + x    &  3
		\\
		& x^{4} + x^{3} y + x^{3}  + y^{3}  + x   &  3
		\\
		& x^{3} y + x^{2} y^{2} + x^{2} y  + y^{3}  + x   
		&  4 \\
		& x^{3} y + x^{2} y^{2} + x^{3}  + y^{3}  + y^{2}  + x 
		&  7 \\
		\hline
	\end{array}
	$$
	
	$$
	\begin{array}{|l|l|c|}
		\hline
		(q,m)   &  \multicolumn{1}{|c|}{C}  & N \\
		\hline
		(2,3) & y^2+( x^{2} + x + 1 )\,y= x^{7} + x^{6} + x^{5} + x^{4} + x^{3} + x + 1  & 1\\
		& y^2+( x^{4} + x + 1 )\,y= x^{8} + x^{5} + x + 1  & 2\\
		& y^2+( x^{4} + x )\,y= x^{8} + x^{7} + x^{5} + x   & 2\\
		& y^2+ x \,y= x^{7} + x^{2} + x   & 2\\
		& y^2+( x^{4} + x^{2} )\,y= x^{8} + x^{4} + x^{2} + x   & 2\\
		& y^2+( x^{4} + x^{2} + 1 )\,y= x^{2} + x + 1   & 2\\
		& y^2+( x^{4} + x + 1 )\,y= x^{6} + x^{5} + x^{4} + x^{3} + 1   & 2\\
		& y^2+( x^{2} + x )\,y= x^{7} + x^{6} + x^{5} + x   & 3\\
		& y^2+y= x^{7}  & 3\\
		& y^2+( x^{2} + x + 1 )\,y= x^{7} + x^{6} + x^{5} + x^{2} + x + 1  & 3\\
		& y^2+y= x^{7} + x^{6} + x^{4} + 1  & 3\\
		& y^2+x \,y= x^{7} + x^{6} + x^{5} + x  & 4\\
		& y^2+( x^{4} + x + 1 )\,y= x^{6} + x^{4} + x^{3} + x^{2} + x + 1  & 4\\
		& y^2+( x^{3} + 1 )\,y= x^{7} + x^{4}  & 4\\
		& y^2+( x^{4} + x^{2} )\,y= x^{4} + x  & 4\\
		&  y^2+( x^{4} + x + 1 )\,y= x^{8} + x^{4} + x^{2} + x  & 4\\
		&  y^2+( x^{4} + x^{2} + 1 )\,y= x^{4} + x  & 6\\
		& y^2+( x^{4} + x + 1 )\,y= x^{6} + x^{5} + x^{3} + x  & 6\\
		&  &     \\[-8pt]
		& x^{4} + x^{2} y^{2} + y^{4} + x^{2} y  + x y^{2}  + x   &  3 \\
		& x^{2} y^{2} + x^{3}  + y^{3}  + x^{2}  + x y  + x   &  4 \\
		& x^{4} + x y^{3} + x^{2} y  + y^{2}  + x &  4 \\
		& x^{3} y + y^{4} + x^{3}  + x    &  5 \\
		& x^{2} y^{2} + y^{4} + x^{3}  + x^{2} y  + x y^{2}  + x   &  5 \\
		\hline
	\end{array}
	$$
	
	$$
	\begin{array}{|l|l|c|c|}
		\hline
		(q,m)   &  \multicolumn{1}{|c|}{C}  & N \\
		\hline   
		(2,4) & x^{3} y + x^{2} y^{2} + x^{3}  + y^{3}  + y^{2}  + x  &  7 \\
		\hline
		(2,5) &  y^2+( x^{4} + x + 1 )\,y= x^{8} + x^{6} + x^{5} + x^{4} + x^{3} + x^{2}   & 4 \\
		& y^2+( x^{4} + x^{2} + 1 )\,y= x^{4} + x   & 6 \\
		&  &     \\[-8pt]
		& x^{4} + x^{2} y^{2} + y^{4} + x^{2} y + x y^{2}  + x^{2} 
		+ x y  + y^{2}  + 1   &  0 \\
		\thickhline
		(3,2) & y^2 = x^{7} + 2 x^{6} + x^{5} + x^{4} + x^{3} + 2 x^{2} + 1 & 5 \\
		& y^2 = x^{8} + 2 x^{5} + 2 x^{4} + 2 x^{2} + 2 x  & 6 \\
		& y^2= 2 x^{7} + 2 x^{4} + 2 x^{3} + 2 x^{2} + 1  & 6\\
		& y^2 = x^{7} + x^{6} + 2 x^{5} + x^{4} + x^{3} + 2 x^{2} + 2 x & 6 \\
		& y^2 = x^{7} + 2 x^{6} + 2 x^{5} + x^{3} + x^{2} + 2 x + 1  & 7 \\
		& y^2 = x^{8} + 2 x^{7} + x^{6} + 2 x^{3} + 2 x^{2} + 1  & 7 \\
		& y^2 = x^{8} + 2 x^{6} + x^{4} + 2 x^{3} + 2 x^{2} + x + 1  & 8 \\
		&  &     \\[-8pt]
		&  x^{4} - x^{2} y^{2} - y^{4} +  x^{3}  - x    & 1 \\
		&  x^{4} +  x^{2} y^{2} - y^{4} +  x^{3}  - x    &  1 \\
		&  x^{4} - x^{3} y - y^{4} +  x^{3}  - x y  - x    &  1 \\
		&  x^{4} - x^{3} y - y^{4} +  x y^{2}  - x    & 2 \\
		&  x^{4} +  x^{3} y - y^{4} +  x^{3}  - x^{2} y  +  x y^{2} - x    &  2 \\
		& x^{4} + x^{3} y + x^{2} y^{2} + y^{4} -  x^{2}  + x y  + x    &  2 \\
		&  x^{4} +  x^{3} y - x y^{3} +  x^{3}  +  x^{2} y  - y^{2} - x    &  2 \\
		&  x^{4} +  x^{3} y +  x^{2} y^{2} - x y^{3} - x^{3} - x^{2} - y^{2}  - x    &  2 \\  
		& y^{4} -  x^{3}  + x    &  4 \\
		& x^{3} y + y^{4} -  x y^{2}  + x    &  5 \\
		&  x^{2} y^{2} - y^{4} +  x^{3}  - x    &  10 \\
		\hline
		(3,3) & y^2 = 2 x^{8} + 2 x^{7} + x^{6} + 2 x^{5} + x^{4} + 2 x^{2} + x + 2 & 3 \\
		& y^2= x^{8} + 2 x^{7} + 2 x^{6} + 2 x^{5} + x^{2} + x  & 4 \\
		& y^2 = x^{7} + x^{6} + 2 x^{5} + x^{4} + x^{3} + x^{2} + 2 x & 4 \\
		& y^2= x^{8} + x^{4} + 2 x^{2} + 1  & 4 \\
		& y^2 = x^{8} + 2 x^{7} + 2 x^{5} + 2 x^{4} + x^{3} + x + 1  & 8 \\
		\thickhline
		\hline
	\end{array}
	$$
	
	$$
	\begin{array}{|l|l|c|c|}
		\hline
		(q,f)   &  \multicolumn{1}{|c|}{C}  & N \\
		\hline   
		(4,2) & y^2+( x^{2} + x + 1 )\,y= x^{7} + x^{6} + x^{5} + x^{3} + x^{2} + x  &  7 \\
		&y^2+( x^{4} + x^{2} + 1 )\,y= x^{5} + x^{2}    &  8 \\
		& y^2+( x^{4} + x^{2} + x + 1 )\,y= x^{5} + x^{3} + x^{2} + x   &  9 \\
		&  &     \\[-8pt] 
		&  \alpha \, x^4+\alpha \, x^3\, y+\alpha \, x^2\, y^2+x+y^4  & 1 \\
		& \alpha \, x^4+\alpha^2  \,x^3 \, y+\alpha^2  \,x^2  \,y^2+x +y^4  & 1 \\
		& \alpha^2 \, x^4+\alpha \, x^3 y+\alpha \,x \,y^3+x +y^4  & 2 \\
		& \alpha \, x^4+\alpha^2 \,x^3 \,y+\alpha^2 \, x \, y^3+x+y^4  & 2 \\
		& \alpha^2\, x^4+x^3\, y+\alpha \, x^2 \,y^2+\alpha^2 \,x^2 \,y +x \,y^3+\alpha \, x \,y^2 +x \, y +x +y^2   & 2 \\
		& x^4+x^3+\alpha \, x^2\, y^2+\alpha \, x^2\, y+x\, y^3+\alpha^2\, x\, y^2+x\, y+x+y^2 & 2 \\
		& x^{3} y + x^{2} y^{2} + x^{3}  + y^{3}  + y^{2}  + x  &  7 \\
		& x^{4} + x y^{3} + x^{2} y  + y^{2}  + x   &  14 \\
		& x^{4} + x^{2} y^{2} + y^{4} + x^{2} y  + x y^{2}  + x^{2} 
		+ x y  + y^{2}  + 1   &  14 \\
		\thickhline
		(5,2) & y^2 = x^7+x^5+3\,x^3+x  & 10 \\
		\thickhline
		(9,2) & 
		x^4+y^4+z^4  & 28  \\
		\hline
	\end{array}
	$$
	
	\begin{proof}
		The proof follows the same arguments as before by inspection but, in addition, we must analyze the cases
		not covered in Sutherland's database. To do so, we follow the directions in Section~\ref{intro} to obtain the list of all candidate Weil polynomials
		for (potential) DS-curves
		of genus $g=3$ attached to the admissible pairs $(q,m)$. If the list is empty, we are done. It happens to be so in all the under-construction cases of Sutherland's database, except for the
		genus-$3$ cases: hyperelliptic with $(q,m)=(4,2)$, and non-hyperlliptic with  $(q,m)=(7,2)$.
		
		As for the case hyperelliptic with $(q,m)=(4,2)$, we proceed to build the census of all (isomorphism classes of) 
		hyperelliptic curves over $\F_4$. To this end, we adapt the code provided by Xarles in \cite{Xarles_github} for hyperlliptic curves over $\F_2$. There are $2162$ isomorphism classes of hyperelliptic curves over $\F_4$, of which three are DS-curves for $\F_{4^2}/\F_{4^{\phantom{2}}}\!\!$; in fact, these three DS-curves can be defined over $\F_2$.
		
	With regard to the non-hyperelliptic case with  $(q,f)=(7,2)$, we proceed following an argument suggested to us by Howe. By using the strategy in Section  \ref{intro}, one finds that there is a unique candidate Weil polynomial, namely:
		$$
		L(t)= (t^2+5\,t+7)(t^4-13\,t^2+49)\,.
		$$
		The corresponding real Weil polynomial of turns out to be:
		$$
		h(x) = P_1(x)P_2(x) = (x+5) (x^2-27)\,.
		$$
		Recall that the roots of the real Weil polynomial $h(x)$
		are $\mu_i=\pi_i+\overline{\pi}_i$, where $\pi_i$ are the roots of $L(t)$.
		Since the resultant of $P_1(x)$ and $P_2(x)$
		equals $-2$, we can
		apply Theorem 1 and Theorem 2 in \cite{Howe}.
		Thus, if there is a genus-3 curve $C$ with the given Weil polynomial, then it must be a double cover of a 
		curve~$D$ such that either:
		\begin{itemize}
			\item[(a)] $D$ is a genus-$2$ curve with Weil polynomial $t^4 - 13 t^2 + 49$, or
			\item[(b)] $D$ is a genus-$1$ curve with Weil polynomial $t^2 + 5 t + 7$.
		\end{itemize}
		
		Case (a) does not work, because there is no genus-$2$ curve with Weil polynomial $t^4 - 13 t^2 + 49$. One can check this either by using the method in Section \ref{intro} again or by the Theorem on page~335 of \cite{Nart}.
		
		So we must be in case (b). Note that $\#D(\F_7) = 13$. We can ask how many of these rational points of $D$ split in the double cover $C \longrightarrow D$, how many ramify, and how many are inert. 
		Since we have $\#C(\F_7) = 13$ and $\#C(\F_{49}) = 13$, no rational points of $D$ are inert, so every rational point either splits or is ramified in the double cover $C\longrightarrow D$. If we let $S$ be the number of split points and $R$ be the number of ramified points, then
		$$
		\begin{array}{l@{\,=\,}l}
			S + R & \#D(\F_7) = 13 \\
			2S + R & \#C(\F_7) = 13,
		\end{array}
		$$
		so $S = 0$ and $R = 13$. But from the Riemann–Hurwitz formula, we see that only~$4$ geometric points of $D$ ramify. Thus, (b) cannot hold either.
		
		We reach to the conclusion that such quartic curve over $\F_7$ with the above Weil polynomial does not exist and this completes the classification of DS-curves of genus $3$ over finite fields.
	\end{proof}
	
	\begin{remark}
		For genus $g\geq 4$, to list the DS-curves over finite fields is an extremely laborious task. For genus $g=4$, Xarles has obtained the list of isomorphism classes of hyperellipitc curves over $\F_2$, but beyond that the numbers of isomorphism classes become large. 
    One is lead to search directly for the candidate real Weil polynomials $h(x)$ as explained in Section \ref{intro} or alternatively use \cite{Ked22}, and then apply Serre's and Howe-Lauter's criteria \cite{HL0307} to discard some candidates. See \cite{Lar21} for a list of DS $4$-genus candidate polynomials (the list has been debugged
        by Howe using the IsogenyClasses.magma \cite{How} package).
	\end{remark}

 \begin{remark}
We want to thank E.Howe for drawing our attention to notice the following fact. For every odd $q$,
one can find a hyperelliptic DS-curve of genus at most $(q^2-3)/2$. His construction goes as follows. Using Lagrange interpolation, construct a polynomial $f$ in $\F_q[x]$ of odd degree at most $q^2-1$ such that:
\begin{itemize}
    \item [(1)] for every $z$ in $\F_q$, the value $f(z)$ is a square in $\F_q$, and
   \item[(2)] for every $z$ in $\F_{q^2}$ that is not in $\F_q$, the value $f(z)$ is a nonsquare in $\F_{q^2}$.
\end{itemize}
The polynomial $f$ might have square factors; let $g$ be the polynomial obtained by dividing out all the square factors from $f$. Then $y^2 = g$ provides a hyperelliptic curve of genus at most 
$(q^2-3)/2$, and it has the same number of points over $\F_{q^2}$ as it does over $\F_q$.

This gives a DS-curve for $\F_{q^2}/\F_{q}$. 
Much simpler is to get hyperelliptic DS-curves for $\F_{q^3}/\F_{q}$, although the genus is larger: choose a nonsquare $n$ in $\F_q$, and let $C$ be the curve $y^2 = f$ with 
$f = x^{q^3} - x + n$. It is easy to check that $f$ has nonzero discriminant, and that $C$ has exactly one rational point over $\F_q$ and exactly one rational point over $\F_{q^3}$.     
 \end{remark}

	\section{Deligne-Lusztig curves}
	
	In the middle 70's, 
	Deligne and Lusztig 
	were able to give an explicit description of the irreducible representations of the  semi-simple finite 
	groups of Lie type~\cite{deligne-lusztig}. 
	These representations can be read from the $\ell$-adic cohomology of the so-called  
	Deligne-Lusztig algebraic varieties which are defined over finite fields. In this section, we make the observation that, in the one-dimensional case, Deligne-Lusztig curves turn out to be DS-curves.

\subsection{Hermitian curves (type ${}^2\!A_2$).}
	
Here $q$ denotes a square 
 prime-power, and $q_0=q^{1/2}$. The hermitian curve is defined by the affine equation:
	$$
	C\colon 
x^{q_0+1}+y^{q_0+1}+z^{q_0+1}=0\,.
	$$
	It is an optimal curve (in fact, it is a maximal curve) of genus $g=(q-q_0)/2$ over $\F_q$. Its Weil polynomial is
	$$
	L(t) = (t+q_0)^{2g}\,.
	$$
	An easy computation shows that
	$$
	\# C(\mathbb{F}_q ) = 
	\# C(\mathbb{F}_{q^2} ) =
	q_0^3+1 = q+1 + 2 g q^{1/2}\,.
	$$
	
	\subsection{Suzuki curves
		(type ${}^2\!B_2$)}
	Here, $q=2^{2e+1}$ and $q_0=2^e$ for $e\geq 1$.
	The Suzuki curve is defined by the affine equation:
	$$
	C \colon y^{q}-y = x^{q_0} (x^q -x) \,.
	$$
	It has genus $g=q_0(q-1)$ and its Weil polynomial is
	$$
	L(t) = \left(t- q^{1/2}\frac{-1+i}{\sqrt{2}}\right)^{g}
	\left(t- q^{1/2}\frac{-1-i}{\sqrt{2}}\right)^{g}
	= (t^2+2q_0 t+q)^g
	\,.
	$$
	An easy computation shows that
	$$
	\# C(\mathbb{F}_q ) = 
	\# C(\mathbb{F}_{q^2} ) =
	\# C(\mathbb{F}_{q^3} ) =
	q^2+1\,.
	$$
	
	\subsection{Ree curves (type ${}^2G_2$)}
	Now, we take $q=3^{2s+1}$, $q_0=3^s$ for $s\geq 1$.
	The Ree curve is given by the two equations
	$$
	C \colon  \ y^q-y = x^{q_0}(x^q-x) \ \, , \ \  z^q-z = x^{q_0}(y^q-y)\,.
	$$
	It has genus 
	$g = \frac{3}{2}q_0(q-1)(q+q_0+1)$. Its Weil polynomial is
	$$
	L(t) = (t^2+q)^{\frac{1}{2}q_0(q-1)(q+3q_0+1)}
	(t^2+3q_0t+q)^{q_0(q^2-1)}\,.
	$$
	One readily checks that
	$$
	\# C(\mathbb{F}_q) = 
	\# C(\mathbb{F}_{q^2}) = 
	\# C(\mathbb{F}_{q^3}) = 
	\# C(\mathbb{F}_{q^4}) = 
	\# C(\mathbb{F}_{q^5}) = 
	1+ q^3 \,.
	$$
	
	\subsection{Drinfeld curve} Drinfeld inspired the general construction of Deligne-Lusztig varieties from the so called Drinfeld curve (see \cite{Bonafe}):
	$$C:y^q-y=z^{q+1}$$
	with $q$ a prime power. This is a DS-curve when $q$ is odd 
    since $\# C(\mathbb{F}_q)=\# C(\mathbb{F}_{q^2})=1+q$. Indeed, the curve has a unique point at infinity defined over $\F_q$, and if $(y,z)\in C(\F_{q^2})$ then $z^{q+1}\in\mathbb{F}_q$ thus $y^q-y=y^q+y-2y\in\mathbb{F}_q$, but $y^q+y$ belongs to $\mathbb{F}_q$, therefore $-2y\in\mathbb{F}_q$ which implies $y\in \F_q$ since $q$ is odd. Thus $y^q-y=0$ which implies $z=0$. 
	Hence the points of $C$ defined over $\F_{q^2}$ are precisely the point at infinity and $(y,0)$ for every $y\in \F_q$.
	
	
\section{$M$-torsion Carlitz curves}
	
	Carlitz initiated  the study of functions fields that play an analogous role to that of cyclotomic number fields in algebraic number theory
	(see \cite{carlitz1}, \cite{carlitz2}).  In this section, we shall deal with the projective no-singular curves attached to the Carlitz modules. Our aim is to point out that Carlitz curves are a good source of DS-curves. First, we recall their definition.
	We refer to \cite{Rosen2002}, \cite{conradcarlitz}, 
	\cite{gebhardt2002constructing}, and
	\cite{bamunoba2014arithmetic} for detailed  expositions on the arithmetic of Carlitz extensions.

	Let $M\in \F_q[t]$ be a monic polynomial of degree $\geq 1$.  The $M$-torsion Carlitz module
	$$
	\Lambda_M = \{ \gamma \in \overline{\F_q(t)} \colon [M](\gamma)= 0\} = \langle \lambda_M \rangle
	$$
	is a finite $1$-dimensional $\F_q[t]$-module via the Carlitz action determined by recursion
	and linearly from the rules: $$[t](x)=x^q+t x\,,\quad  [t^n](x)=[t]([t^{n-1}](x))
	\text{ for $n\geq 2$, and\ }  [1](x)=x\,.$$

	Let $K_M=\F_q(t,\lambda_M)$ and consider the Carlitz extension $K_M/\F_q(t)$  attached to $M$. The field $K_M$ produces an abelian extension 
    over $\F_q(t)$ unramified outside the primes dividing $M\infty$, where $\infty$ corresponds 
    to the place $1/t$. The Carlitz action induces an isomorphism between $(\F_q[t]/M)^*$ and the Galois group $\operatorname{Gal}(K_M/\F_q(t))$:
	$$
	(\F_q[t]/M)^* \longrightarrow
	\operatorname{Gal}(K_M/\F_q(t))
	\,,\quad
	Q \mapsto \sigma_Q\colon \lambda_M\mapsto [Q](\lambda_M)\,.
	$$
	Let $\Phi_{M}(x)$ be the
	$M$-th Carlitz polynomial
	defining the extension $K_M/\F_q(t)$; that is,
	$$
	\Phi_{M}(x) = \frac{[M](x)}{\displaystyle{\prod_{Q|M}} \Phi_{Q}(x)}
	$$
	where $Q$ runs the monic polynomials in $\F_q[t]$ dividing $M$ of degree less than $\operatorname{deg}M$, and $\Phi_1(x)=x$.
	The minimal polynomial of $\lambda_M$ over $\F_q(t)[x]$ is the irreducible Carlitz polynomial
	$\varPhi_M(x)$.
	
	The $M$-torsion field $K_M$ can be regarded as the function field of an algebraic curve defined over~$\F_q$, that we shall denote here by $\mathcal{X}_M$ and call it the Carlitz curve of level $M$. In fact, by Galois theory, for every subgroup 
	$H\subseteq (\F_q[t]/M)^*$ we can consider the non-singular projective curve 
	$\mathcal{X}_M^H$ attached
	to the fixed field $K_M^H$. A result of Weil \cite{Wei48} allows us to compute the zeta function of $\mathcal{X}_M^H$. Indeed, on has
	$$
	\zeta(\mathcal{X}_M^H/\F_p,t) = \frac{P_H(t)}{(1-t)(1-qt)}
	$$
	with
	$$
	P_H(t)= \prod_{\substack{\chi \neq 1\\  \chi_H = 1}} L(\chi,t)
	$$
	where
	$\chi\colon (\F_q[t]/M)^* \longrightarrow \C^*$ runs the non-trivial Dirichlet characters and 
	$\chi_H$ stands for the restriction of $\chi$ to $H$. The coefficients of the polynomials 
	$
	L(\chi,t)=\sum_{n=0}^\infty A(n,\chi) t^n 
	$
	satisfy
	$$
	A(n,\chi) = \sum_{\substack{\operatorname{deg}(f)=n \\  f \text{ monic} }} \chi(f) \,.
	$$
	By the orthogonality relations, one has
	$A(n,\chi)=0$ for $n\geq \operatorname{deg}(M)$. In order to explore the DS property of the curves 
	$\mathcal{X}_M^H$, we need to control
	the number of places $a_d^H$ on $\mathcal{X}_M^H$ of degree $d$.
	
	To this end, one must take into account that all places of $\mathcal{X}_M^H$
	over
	the place $\infty$ of $\F_p(t)$ 
	have degree one. In other words, the place $\infty$
	only contributes to $a_1^H$ (see \cite{Rosen2002}, Chap. 13). In general, for a finite place $\mathfrak{P}$ of degree $d$ in $\mathcal{X}_M^H$ over a place $\mathfrak{p}=(\pi)$, where $\pi\in \F_q[t]$ is an irreducible polynomial of 
	degree~$d'$, one has
	$$
	\# [  (\F_q[t]/M)^*\colon H ] = e_{\pi} f_{\pi} g_{\pi} 
	$$
	where $e_{\pi}$ is the ramification index at $\mathfrak{P}$, 
	$f_{\pi} = [\F_{\mathfrak{P}} \colon \F_{\mathfrak{p}}]
	= [\F_{p^d} \colon \F_{p^{d'}}]
	$ is the residual degree, and 
	$g_{\pi}$ is the number of Galois conjugates of ${\mathfrak{P}}$.
	The residual degree $f_\pi$ satisfies
	$
	d = f_\pi \cdot d'
	$
	and it can be computed as the order of the class $\pi \bmod M$ in the subgroup $H$; that is, $f_\pi$ is the minimum integer $f\geq 1$ such that $\pi^f \bmod M$ belongs to $H$. 
 
        For $d>1$ 
	one has
	$$
	a_d^H = \frac{m}{h} \,.
	\sum_{d'\mid d}
	\left(
	\sum_{
		\substack{
			\deg(\pi)=d'\\
			d=f_\pi d'}}
	\frac{1}{f_\pi e_\pi}
	\right)
	$$
	where $m=\# (\F_q[t]/M)^*$, and  
	$h=\# H$.
	
	It is fairly easy to exhibit instances of $M$ and $H$ that produce DS-curves $\mathcal{X}_M^H$. 
	
	\begin{prop}\label{prop5.1}
		Let $\ell$ be a prime number, and let 
		$H\subseteq (\F_q[t]/M)^*$
		be a subgroup.
		Assume that all irreducible factors of $M\in \F_q[t]$ have degree different form $1$ and $\ell$. Moreover, suppose that
		the classes mod $M$ of all primes $\pi\in \F_q[t]$ of
		degree $1$ and $\ell$ do not have order $\ell$ and $1$ in $H$, respectively. Then, $\mathcal{X}_M^H$ has Diophantine stability for $\F_{q^\ell}/\F_q$.
	\end{prop}
	
	\begin{proof}
		When $d=
		\ell$ is a prime number, the primes $\pi$ of $\F_q[t]$ that can contribute to $a_\ell^H$ have residual degree either $f_\pi=1$ or $\ell$. Thus, such a polynomial $\pi$ must have degree $\ell$ or $1$ and, due to the hypothesis, they are
		coprime with $M$ and hence are unramified in $K_M^H$. 
		Also we have that the order of $\pi$ mod $M$ in $H$ is not $\ell$ for all primes $\pi$ of degree 1 and the order of $\pi$ in $H$ is not $1$ for all primes $\pi$ of degree $\ell$, then $a_\ell^H=0$. Thus, 
		$\mathcal{X}_M^H$ has Diophantine stability for $\F_{q^\ell}/\F_q$.
	\end{proof}

	\begin{prop}\label{prop5.2}
		Let $M\in\mathbb{F}_q[t]$ be a polynomial of degree $m>2$,
		and let $H\subseteq \mathbb{F}_q^*\subseteq (\mathbb{F}_q[t]/M)^*$ be a subgroup.
		Let $k$ be an integer with $2\leq k<m$. Assume that for every prime divisor $\pi$ of $M$ it holds $k\neq deg(\pi) f_{\pi}$. Then, one has $a_k^H=0$. 
	\end{prop}
	\begin{proof}
		Suppose for contradiction that $\mathcal{X}_M^H$ has a place $\mathfrak{P}$ of degree $2\leq k<m$. 
		Since $k\neq 1$ we can assume that $\mathfrak{P}$ is a finite place. Then
		$\mathfrak{P}\cap\mathbb{F}_q[t]=(\widetilde\pi)$, with $\widetilde\pi$ a prime in $\mathbb{F}_q[t]$ of degree $d$ where $k=f_{\widetilde\pi} d$. By hypothesis we must have $\widetilde\pi\nmid M$; that is, $\widetilde\pi$ is an unramified prime, and thus $f_{\widetilde\pi}$ is the minimum power such that $\widetilde\pi^{f_{\widetilde\pi}}\in H$. Hence $\widetilde\pi^{f_{\widetilde\pi}}-u\in M\mathbb{F}_q[t]$ for some $u\in H\subseteq \mathbb{F}_q^*$ and it has degree $f_{\widetilde\pi} d=k<m$. But this is impossible since 
		$M$ has degree $m$ and $\widetilde\pi^{f_{\widetilde\pi}}-u$ is not the zero polynomial. Therefore we must have $a_k^H=0$.
	\end{proof}
	
	\begin{corollary}\label{cor5.2}
		Let $M=\pi^r\in\mathbb{F}_q[t]$ be a prime-power polynomial with $m=deg(\pi)>2$ and $r\geq 1$. Let $H\subseteq \mathbb{F}_q^*\subseteq (\mathbb{F}_q[t]/M)^*$ be a subgroup. Then, $a_k^H=0$ for all $k$ such that $2\leq k<rm$ and $k\neq m$.
	\end{corollary}
	\begin{proof}
		Fix $k$ with $2\leq k<rm$ and $k\neq m$. 
		Since $K_M/\F_q(t)$ is totally ramified at $\pi$, it follows that $f_\pi =1$. Hence, we can apply Proposition \ref{prop5.2} 
		since $k\neq \operatorname{deg}(\pi) f_\pi=m$.
	\end{proof}
	
	Let us illustrate the above results with two examples.

 \vskip 0.2truecm
	
	Example 1. Consider $M=t^4+t+1$ in $\F_2[t]$ and let $H$ be the trivial subgroup.
	The curve $\mathcal{X}_M$ has genus $14$.  
	The numerator of zeta function $\zeta(\mathcal{X}_M/\F_2,T)$ is the polynomial:
	$$
	\begin{array}{ll}
		P(T) = & (4T^{4} - T^{2} + 1) \cdot
		(4T^{4} + 2T^{3} + 3T^{2} + T + 1)^{2}  \cdot    \\[8pt]
		& (16T^{8} + 40T^{7} + 52T^{6} + 50T^{5} + 39T^{4} + 25T^{3} + 13T^{2} + 5T + 1)^{2} \,.
	\end{array}
	$$
	The sequence of number of places begins:
	$$[a_d]=
	[15,0, 0, 1, 0, 5, 30, 30, 60, 45, 210, 345, 690, 1095,\dots] .$$
The values $a_2=a_3=0$ are explained by Corollary \ref{cor5.2}.

  \vskip 0.2truecm
  
	Example 2. Take
	$M = (t^6+t+1)^2$ and $H$ trivial in $\F_2[t]$. The extension $K_M/\F_2(t)$ has degree
	$4032=2^6\cdot 3^ 2 \cdot 7$ and the curve $\mathcal{X}_M$ has genus $19969$. The sequence of number of places begins:
	$$
	[a_d]=[4032,0,0,0,0,1,0,0,0,0,0,0,0,0,\dots]\,.
	$$
	The values $a_k=0$ for $2\leq k<12$ with $k\neq 6$  are explained by Corollary \ref{cor5.2}. By Proposition \ref{prop5.1}, $a_6$ can only be contributed by the unique ramified place $(t^6+t+1)$ and easily one finds $a_6=1$.  

  \vskip 0.2truecm
  
The computations involved in the examples above have been performed with Magma \cite{Bos97} and SageMath \cite{Sage24}.	
	
\section{$M$-torsion Drinfeld curves}
	
	As we have discussed in the previous section, the Carlitz action is defined recursively and linearly by the axioms:
	$$
	\begin{array}{l@{\,=\,}l}
		[1](x) & x  \\[2pt]
		[t](x) & x^q+tx \\[2pt]
		[t^n](x) & [t]([t^{n-1}](x)) \text{ for $n\geq 2$.}
	\end{array}
	$$
	With the aim to attack the Langlands program in positive characteristic, Drinfeld generalized the Carlitz action as follows. Consider a $q^n$-degree polynomial of the form
	\begin{equation}\label{eqDr}
		\mathfrak{h}(x)=u_n x^{q^n}+u_{n-1}
		x^{q^{n-1}}+\ldots+u_1x^q+tx\in\mathbb{F}_q(t)[x]\,,
	\end{equation} 
	and define the Drinfeld action 
	recursively and $\mathbb{F}_q$-linearly by:
	$$
	\begin{array}{l@{\,=\,}l}
		[1]_{\mathfrak{h}}(x) & x  \\[2pt]
		[t]_{\mathfrak{h}}(x) & \mathfrak{h}(x) \\[2pt]
		[t^n]_{\mathfrak{h}}(x) & [t]_{\mathfrak{h}}([t^{n-1}]_{\mathfrak{h}}(x)) \text{ for $n\geq 2$.}
	\end{array}
	$$
    

	For a given polynomial $M\in\mathbb{F}_q[t]$, consider the $n$-rank $M$-torsion Drinfeld module
	$$\Lambda_{M,\mathfrak{h}}:=\{\gamma\in\overline{\mathbb{F}_q(t)}:[M]_{\mathfrak{h}}(\gamma)=0\}\,.$$ 
	One has that $\Lambda_{M,\mathfrak{h}}\cong (\mathbb{F}_q[t]/M)^n$ as $\F_q[t]$-module with the Drinfeld action.
    The image of the embedding
	$$
	\operatorname{Gal}(\F_q(t)(\Lambda_{M,\mathfrak{h}})/\mathbb{F}_q(t))
	\rightarrow 
	\operatorname{GL}_n(\mathbb{F}_q[t]/M)
	$$
	has been studied by Pink and his collaborators (see \cite{pink1}, \cite{pink2}, \cite{pink3}, \cite{pink4}, \cite{pink5}.
	
	Similarly as we did in the Carlitz case, we introduce the Drinfeld polynomials
$$\Phi(x,t)=\Phi_{M,\mathfrak{h}}(x):=\frac{[M]_{\mathfrak{h}}(x)}{\prod_{Q|M}\Phi_{Q,\mathfrak{h}}(x)}\in\mathbb{F}_q[t][x],$$
	where $Q$ runs through the monic polynomials in $\mathbb{F}_q[t]$ dividing the monic polynomial $M\in\mathbb{F}_q[t]$ and  $\operatorname{deg}_x(Q) < \operatorname{deg}_x(M)$.

	From now on, we assume that
	$
	\Phi(x,t)$ is an irreducible polynomial in $\F_q[t][x]$ and denote by
	${\mathcal X}_{\Phi}$
	the Drinfeld curve over $\F_q$ determined by the defining equation 
	$
	\Phi(x,t)=0$. We say that 
 ${\mathcal X}_{\Phi}$ is the $M$-torsion $\mathfrak{h}$-Drinfeld curve of rank $n$. Also, let $K_{\Phi}$ be the function field associated to $\mathcal{X}_{\Phi}$; namely, the quotient field $\mathbb{F}_q(t)[x]/(\Phi(x,t))$.
 
\begin{remark} To control the  ramification divisor of the extension
$\F_q(t)(\Lambda_{M,\mathfrak{h}})/\mathbb{F}_q(t)$ is not an easy task (see Taguchi \cite{Tag}). However, for a given $\Phi$ as above, 
the  ramification divisor of the extension
$K_\Phi/\mathbb{F}_q(t)$
is more manageable using 
valuation theory. 
\end{remark} 

In the next subsections we build some examples of Drinfeld curves with Diophantine Stability.

\subsection{Examples of 
Drinfeld DS-curves of rank 3}	
We show how to impose conditions on the coefficients of the polynomial 
\begin{equation}\label{exam3rank}
		\mathfrak{h_3}(x)=u_3 x^{q^3}+u_{2}
		x^{q^2}+u_1x^q+tx\in\mathbb{F}_q(t)[x]
	\end{equation}
to produce $t$-torsion
Drinfeld curves of rank $3$
 with Diophantine stability with $q=2$. The method presented here can be extended to produce examples of Drinfeld DS-curves for higher rank. As usual for a place $\pi$ of $\mathbb{F}_q(t)$ we denote by $\operatorname{ord}_{\pi}(f)$ or $v_{\pi}(f)$ for $f\in\mathbb{F}_q(t)$ the integer $k$ such that $\pi^k||f$, 
	and denote by $\mathbb{F}_q(t)_{\pi}$ the completion of $\mathbb{F}_q(t)$ at $\pi$  with the natural extension of $\operatorname{ord}_{\pi}=v_{\pi}$ in the completion field.
%
	
	
	\begin{lemma} 
        \label{lemma62}
 	Let $\mathfrak{h_3}(x)=u_3 x^{q^3}+u_{2}
		x^{q^2}+u_1x^q+tx\in\mathbb{F}_q(t)[x]$ with $q=2$. Denote $\Phi(x,t)=\Phi_{t,\mathfrak{h_3}}(x)$ defined by $ [t]_{\mathfrak{h_3}}(x)= x \cdot \Phi_{t,\mathfrak{h_3}}(x)$. Assume that
  $(
  \operatorname{ord}_t(u_3),
  \operatorname{ord}_t(u_2),
  \operatorname{ord}_t(u_1)
  ) = (0,\geq 1,\geq 1)$.
  Then, 
			\begin{enumerate}
			\item $\Phi_{t,\mathfrak{h}_3}(x)\in\mathbb{F}_q(t)[x]$ is irreducible.
			\item Assume that for all  $\pi\in\{1/t,t+1,t^2+t+1\}$ it holds $\operatorname{ord}_{\pi}(u_i)\geq 0$ for $i
   \in\{1,2,3\}$. In addition, suppose that $\gcd(x^4+x,\Phi_{t,\mathfrak{h_3}}(x) \!\!\pmod{t^2+t+1})=1$ and $\gcd(x^2+x+1,\Phi_{t,\mathfrak{h_3}}(x)\!\! \pmod{t+1})\neq x^2+x+1$, then
the $t$-torsion $\mathfrak h_3$-Drinfeld curve $\mathcal{X}_{\Phi}$ of rank $3$ has DS for the extension $\F_4/\F_2$. 
\footnote{An explicit example of $\mathfrak{h}_3$ satisfying all the assumptions in lemma \ref{lemma62} is: $u_1(t)=\frac{t(t^2+t+1)}{(t^3+t+1)}$, $u_2(t)=\frac{t(t+1)^2}{(t^3+t+1)}$ and $u_3(t)=1$.}
\end{enumerate} 
\end{lemma}
\begin{proof} 
 Due to the hypothesis on the coefficients of 
 $\Phi_{t,\mathfrak h_3}(x)=u_3 x^7+u_2 x^3+u_1 x+t$, the first claim is an immediate consequence of Eisenstein's irreducibility criterion applied to the prime $t$.
By considering the Newton polygon of $\Phi_{t,\mathfrak h_3}(x)$
at the place $t$, we observe that $t$ is totally ramified in the extension $K_\Phi/\F_2(t)$ (we refer to \cite[II,\S6]{NeukirchBook} for the properties of Newton polygons).

In order to compare the sets $\mathcal{X}_{\Phi}(\mathbb{F}_2)$ and $\mathcal{X}_{\Phi}(\mathbb{F}_4)$ we need to control the decomposition of the places of degrees 1 and 2: $t$, $t+1$, $t^2+t+1$, $\infty=1/t$ of $\mathbb{F}_2(t)$ in the field extension $K_{\Phi}/\mathbb{F}_2(t)$.
		
Since the place $t$ is totally ramified in $K_{\Phi}/\mathbb{F}_2(t)$, 
it does not produce points 
in $\mathcal{X}_{\Phi}(\mathbb{F}_4) 
\setminus
\mathcal{X}_{\Phi}(\mathbb{F}_2)
$.
Similarly, considering the Newton polygon of 
$\Phi_{t,\mathfrak h_3}(x)$
at the place $1/t$ and
taking under consideration our assumptions on $\operatorname{ord}_{1/t}$
for the coefficients $u_i$, we observe that the place $1/t$ produces a unique point at infinity
in $\mathcal{X}_{\Phi}$ which is defined over $\F_2$.
		
Now consider the place $\pi=t+1$ of $\mathbb{F}_2(t)$.		
Since $\operatorname{ord}_{t+1}(u_i)\geq 0$, we can consider the reduced polynomial
$g(x):=\Phi_{t,\mathfrak{h_3}}(x) 
\pmod {t+1}\in\mathbb{F}_2[x]$.
In order that the place $t+1$ does not provide points in 
$\mathcal{X}_{\Phi}(\mathbb{F}_4) 
\setminus
\mathcal{X}_{\Phi}(\mathbb{F}_2)$ we impose that the factorization over $\mathbb{F}_4[x]$ of the above polynomial does not appear any new linear factor; i.e. we need 
$x^2+x+1\nmid g(x)$.
		
Finally take the place $\pi=t^2+t+1$ of $\mathbb{F}_2(t)$, and assume $\operatorname{ord}_{t^2+t+1}(u_i)\geq 0$. Now we consider $g(x):=\Phi_{t,\mathfrak{h_3}}(x) 
    \pmod {t^2+t+1}\in\mathbb{F}_4[x]$,
and we want to impose that $g(x)$ has no roots in $\F_4$; that is to say, $\gcd(x^4+x,g(x))=1$. Putting altogether, the hypothesis on the valuations of the coefficients of $\mathfrak{h}_3(x)$ ensure that $\mathcal{X}_{\Phi}(\mathbb{F}_4) =
\mathcal{X}_{\Phi}(\mathbb{F}_2)$.
	\end{proof}

\subsection{Examples of 
Drinfeld DS-curves of arbitrary large rank}

In this subsection we use base change techniques to construct 
Drinfeld DS-curves defined over $\F_q$ of rank $n\geq 1$. 
	
To this end we consider the Drinfeld action on $\mathbb{F}_q[t]$ given by $\mathfrak{h}_{q,n}(x)=x^{q^n}+t x$.
Observe that we have the following equality:
	
\begin{equation}\label{basechange}
		[t]_{\mathfrak{h}_{q,n}}(x)
=x^{q^n}+tx =  [t]_{\mathfrak{h}_{q^n,1}}=[t](x).
	\end{equation}
In other words, for $\mathfrak{h}_{q,n}(x)=x^{q^n}+t x$,
the Drinfeld action on $\F_q[t]$ is the same as the Carlitz action on $\F_{q^n}[t]$. 
By linearity, one has  that for every $M\in \F_q[t]$ it holds
$
[M]_{\mathfrak{h}_{q,n}} (x)=
[M]_{\mathfrak{h}_{q^n,1}}(x)
$. From now on in this section, we assume that $M$ is a polynomial 
in $\F_{q}[t]$ that with the property that it factorizes in $\F_{q^n}[t]$ exactly as it does over $\F_{q}[t]$. Then, we have	
$$
\Phi_{M,\mathfrak{h}_{q,n}}(x)=\frac{[M]_{\mathfrak{h}_{q,n}}(x)}{\prod_{Q|M}\Phi_{Q,\mathfrak{h}_{q,n}}(x)}
=
\frac{[M]_{\mathfrak{h}_{q^n,1}}(x)}{\prod_{Q|M}\Phi_{Q,\mathfrak{h}_{q^n,1}}(x)} =
\Phi_{M,\mathfrak{h}_{q^n,1}}(x)
\in \F_q[t][x]\,.
$$
 	
Given such $M\in \F_q[t]$ and $n\geq 1$, 
for each $i\mid n$ we consider 
the function field: $$K_{M,i}=\mathbb{F}_{q^i}(t)[x]/(\Phi_{M,\mathfrak{h}_{q^n,1}}(x)),$$
and the corresponding curve $\mathcal{X}_{{M},i}$
defined over $\F_{q^i}$. Notice that
$\mathcal{X}_{M,n}$ is the $M$-torsion Carlitz curve over $\mathbb{F}_{q^n}$ discussed in the previous section,
and that $\mathcal{X}_{M,1}$ is the 
$M$-torsion $\mathfrak h_{q,n}$-Drinfeld curve of rank~$n$ defined over $\F_q$.
The different curves $\mathcal{X}_{{M},i}$ are obtained by constant field extensions from~$\mathcal{X}_{{M},1}$. Thus by \cite[Proposition  8.19]{Rosen2002} one has the following relation on the number of $k$-degree places:
	
	\begin{equation}\label{eqRosen1}
		a_k(\mathcal{X}_{{M},n})=\frac{1}{k}\sum_{d|k}\mu (d)a_1(\mathcal{X}_{{M},nk/d})
	\end{equation}
	where $\mu$ denotes the M\"oebius function.
	
	\begin{lemma}\label{descent}
	Keep the notations as above and let $n=\ell$ be a  prime number.
	If $a_k(\mathcal{X}_{{M},\ell})=0$ for
 some $k$ with $\ell \mid k$, then $a_{\ell k}(\mathcal{X}_{{M},1})=0$. 
	\end{lemma}
	\begin{proof}
		We know that $0=a_k(\mathcal{X}_{{M},\ell})=\frac{1}{k}\sum_{d|k}\mu(d)a_1(\mathcal{X}_{{M},\ell k/d})$, and want to compute
		\begin{equation}\label{doseq}
			a_{\ell k}(\mathcal{X}_{{M},1})=\frac{1}{\ell k}\sum_{d|\ell k} \mu(d) a_1(\mathcal{X}_{{M},\ell k/d}).
		\end{equation}
		
We can (and do) write $k=\ell^m k'$ with $k'$ coprime with $\ell$ and $m\geq 1$ by assumption. 
We consider the set $\mathcal{S} = 
\{ d\in \Z_{\geq 1} \colon d\mid \ell k \,,\, d\nmid k \}.
$ Notice that if $d\in \mathcal S$, then one has $d=\ell^{m+1}k''$ with $k''|k'$.  We rewrite (\ref{doseq}) as follows:
		$$(\ell k)a_{\ell k}(\mathcal{X}_{{M},1})=\sum_{d|k}\mu(d)a_1(\mathcal{X}_{M,\ell k/d})+\sum_{d\in\mathcal{S}}\mu(d)a_1(\mathcal{X}_{{M},\ell k/d})\,.$$
	The first summand is zero by hypothesis and the second as well due to $m+1\geq 2$ which implies $\mu(d)=0$ for all $d\in\mathcal{S}$. 
	\end{proof}
	\begin{cor}\label{cor6.3} Let $\ell$ be a prime. Let $M\in\mathbb{F}_q[t]$ of degree $>\ell$ and irreducible over $\F_{q^\ell}[t]$.
    Then $\mathcal{X}_{M,1}(\mathbb{F}_{q^{\ell}})= \mathcal{X}_{M,1}(\mathbb{F}_{q^{\ell^2}})$, and thus $\mathcal{X}_{M,1}$  is a 
    DS-curve for $\F_{q^{\ell^2}}/\F_{q^\ell}$. 	
	\end{cor}
	\begin{proof}
    From Corollary \ref{cor5.2}, we have that the 
    $M$-torsion  Carlitz curve 
    $\mathcal{X}_{M,\ell}$ over $\mathbb{F}_{q^{\ell}}$ satisfies $a_{\ell}(\mathcal{X}_{{M},\ell})=0$. Applying Lemma \ref{descent} we get
		$a_{\ell^2}(\mathcal{X}_{{M},1})=0$.
  Therefore by Proposition \ref{prop2.5}, we find 
		$\mathcal{X}_{{M},1}(\mathbb{F}_{q^{\ell}})=\mathcal{X}_{{M},1}(\mathbb{F}_{q^{\ell^2}})$.
	\end{proof}
	\begin{example}
		Take $q=2$, $n=\ell =2$, and $M=t^3+t+1$. Observe that $M$ is irreducible in $\mathbb{F}_{2^2}[t]$. One has
		$$\Phi_{M,\mathfrak{h}_{2,2}}(x)=\Phi_{M,\mathfrak{h}_{4,1}}(x)=x^{63}+(t^{16}+t^4+t)x^{15}+(t^8+t^5+t^2)x^3+(t^3+t+1)\,.$$
	From Corollary \ref{cor5.2}, we get that the Carlitz curve $\mathcal{X}_{{M},2}/\F_{4}$ 
 satisfies $a_2(\mathcal{X}_{{M},2})=0$. By Corollary~\ref{cor6.3}, we obtain
	$a_4(\mathcal{X}_{{M},1})=0$.  Therefore
		$$\mathcal{X}_{{M},1}(\mathbb{F}_{2^2})=\mathcal{X}_{{M},1}(\mathbb{F}_{2^4})$$
		thus the Drinfeld curve $\mathcal{X}_{{M},1}/\F_{2}$ is a DS-curve for the extension $\F_{2^4}/\F_{2^2}$.
	\end{example}

   \bibliographystyle{alpha}

\begin{thebibliography}{99}
		\bibitem[Bam14]{bamunoba2014arithmetic} A.S. Bamunoba: ``On some properties of Carlitz cyclotomic polynomials". J. Number Theory 143 (2014), 102--108.
		\bibitem[Bon11]{Bonafe} C. Bonnaf\'e: 
  ``Representations of $SL_2(\mathbb{F}_q)$, Algebra and Applications, vol. 13,
		Springer-Verlag, London, 2011.
\bibitem[Bos97]{Bos97} 
W. Bosma, J. Cannon, C. Playoust: ``The Magma algebra system. I. The user language", J. Symbolic Comput., 24 (1997), 235–265.
		\bibitem[Car35]{carlitz1}  L. Carlitz: ``On certain functions connected with polynomials in a Galois field", Duke Math J. 1 (1935),
		137–-168.
		\bibitem[Car38]{carlitz2} L. Carlitz: ``A class of polynomials", Trans. Amer. Math. Soc. 43 (1938), 167–-182.
		\bibitem[Con00]{conradcarlitz} K. Conrad: ``Carlitz extensions". Expository papers in Algebraic Number Theory in personal web page https://kconrad.math.uconn.edu/
		\bibitem[DL76]{deligne-lusztig} P. Deligne, G. Lusztig: ``Representations of reductive groups over finite
		fields", Ann. of Math. 103 (1976), 103–-161.
		\bibitem[Geb02]{gebhardt2002constructing} M. Gebhardt: ``Constructing function fields with many rational places via the Carlitz module". Manuscripta math. 107, 89-–99 (2002). https://doi.org/10.1007/s002290100226

	\bibitem[How]{How} E.W. Howe, in: 
https://ewhowe.com/papers/paper35.html
  
		\bibitem[How12]{Howe} E.W. Howe: ``New methods for bounding the number of points on curves over finite fields", pp. 173–212 in:
		Geometry and Arithmetic (C. Faber, G. Farkas, and R. de Jong, eds.), European Mathematical Society, 2012.

		\bibitem[HL12]{HL0307} E.W. Howe, K. E. Lauter: “New methods for bounding the number of points on curves over finite fields”, pp. 173–212 in “Geometry and arithmetic” (C. Faber, G. Farkas, R. de Jong, eds.), Eur. Math. Soc., 2012. http://dx.doi.org/10.4171/119-1/12.
	\bibitem[Kat76]{Kat76}
        N. Katz: ``An overview of Deligne's proof of the Riemann hypothesis for varieties over finite fields", Proc. Symp. Pure Math. 28 (1976), 275--305.
  	
        \bibitem[Ked22]{Ked22}
        K. Kedlaya (et al.): ``Isogeny classes of abelian varieties over finite fields in the LMFDB", in Arithmetic Geometry, Number Theory, and Computation, Simons Symposia, Springer, 2022, 375--448.

        \bibitem[Lar21]{Lar21} J-C. Lario, in: https://web.mat.upc.edu/joan.carles.lario/DS.html\#

        \bibitem[Lau00]{Lau00}
        K. Lauter: “Zeta functions of curves over finite fields with many rational points”, pp. 167–174 in “Coding theory, cryptography and related areas (Guanajuato, 1998)” (J. Buchmann, T. Hoholdt, H. Stichtenoth, H. Tapia-Recillas, eds.), Springer, 2000. 
        
        \bibitem[MNH02]{Nart} D. Maisner, E. Nart (with an appendix by E. W. Howe): ``Abelian surfaces over finite fields as Jacobians", Experiment. Math. 11 (2002) 321–-337.
		\bibitem[Neu02]{NeukirchBook} J. Neukirch: ``Algebraic Number Theory", Springer Verlag (2002).
		\bibitem[Pin06a]{pink1}R. Pink: ``The Galois Representations Associated to a Drinfeld Module in Special Characteristic, I: Zariski Density"
		J. Number Theory 116 (2006), no. 2, 324--347.
		\bibitem[Pin06b]{pink2} R. Pink: ``The Galois Representations Associated to a Drinfeld Module in Special Characteristic, II: Openness"
		J. Number Theory 116 (2006) no. 2, 348--372.
		\bibitem[PT06]{pink3} R. Pink, M. Traulsen:`` The Galois representations associated to a Drinfeld module in special characteristic. III. Image of the group ring". J. Number Theory 116 (2006), no. 2, 373--395. 
		\bibitem[PR09a]{pink4} R. Pink, E. R\"utsche: ``Adelic openness for Drinfeld modules in generic characteristic". J. Number Theory 129 (2009), no. 4, 882--907. MR2499412 Add to clipboard 
		\bibitem[PR09b]{pink5} R.Pink, E. R\"utsche: ``Image of the group ring of the Galois representation associated to Drinfeld modules". J. Number Theory 129 (2009), no. 4, 866--881.
       
        \bibitem[Rob64]{Rob64}
        R.M. Robinson: “Algebraic equations with span less than 4”, Math. Comp. 18 (1964), 547–-559. http://dx.doi.org/10.2307/2002941

        
		\bibitem[Ros02]{Rosen2002} M. Rosen: ``Number theory in function fields". Graduate Texts in Mathematics, 210. Springer-Verlag, New York, 2002
	\bibitem[Sag24]{Sage24}
        SageMath, the Sage Mathematics Software System,
        The Sage Developers, 2024, https://www.sagemath.org. 
	\bibitem[Ser20]{Ser20}
        J.-P. Serre: "Rational points on curves over finite fields". Documents mMth. vol. 18. Societé Mathématique de France, 2020.

        \bibitem[Smi84]{Smi84}
        C.J. Smyth: “Totally positive algebraic integers of small trace”, Ann. Inst. Fourier 34 (1984), 
        1--28. 
        
        \bibitem[Sut]{sutherlandwebpage} A.V. Sutherland, in: https://math.mit.edu/~drew/avff/
		\bibitem[Tag91]{Tag} Y. Taguchi: ``Semisimplicity of the Galois representations attached to Drinfeld modules over fields of “finite characteristics”".
		Duke Math. J. 62(3): 593--599 (1991).
    	\bibitem[Vri]{Vri} B. Vrioni: "A census for curves and surfaces with diophantine stability over finite fields". Tesi doctoral, UPC, Facultat de Matemàtiques i Estadística, 2021. DOI 10.5821/dissertation-2117-360914 . in: https://upcommons.upc.edu/handle/2117/360914?show=full
  	\bibitem[Wei48]{Wei48} A. Weil: ``Sur les courbes algébriques et les variétés qui s'en déduisent”". Hermann, Paris (1948).
	\bibitem[Xar]{Xarles_github} X. Xarles in: https://github.com/XavierXarles/Censusforgenus4curvesoverF2

	\end{thebibliography}
	
\end{document}